\documentstyle[10pt,amscd,xypic,amssymb]{amsart}
\xyoption{all}
\CompileMatrices

\newcommand{\nc}{\newcommand}

\makeatletter
\@addtoreset{equation}{section}
\makeatother

\newenvironment{proof}{{\noindent \textbf{Proof}\,\,}}{\hspace*{\fill}$\Box$\medskip}

\newtheorem{theorem}[subsection]{Theorem}

\newtheorem{proposition}[subsection]{Proposition}
\newtheorem{lemma}[subsection]{Lemma}

\theoremstyle{definition}

\theoremstyle{remark}

\newtheorem{remark}[subsection]{Remark}

\nc{\fa}{{\mathfrak{a}}}
\nc{\fb}{{\mathfrak{b}}}
\nc{\fg}{{\mathfrak{g}}}
\nc{\fh}{{\mathfrak{h}}}
\nc{\fj}{{\mathfrak{j}}}
\nc{\fn}{{\mathfrak{n}}}
\nc{\fm}{{\mathfrak{m}}}
\nc{\fu}{{\mathfrak{u}}}
\nc{\fp}{{\mathfrak{p}}}
\nc{\fr}{{\mathfrak{r}}}
\nc{\ft}{{\mathfrak{t}}}
\nc{\fsl}{{\mathfrak{sl}}}
\nc{\fgl}{{\mathfrak{gl}}}
\nc{\hsl}{{\widehat{\mathfrak{sl}}}}
\nc{\hgl}{{\widehat{\mathfrak{gl}}}}
\nc{\hg}{{\widehat{\mathfrak{g}}}}
\nc{\chg}{{\widehat{\mathfrak{g}}}{}^\vee}
\nc{\hn}{{\widehat{\mathfrak{n}}}}
\nc{\chn}{{\widehat{\mathfrak{n}}}{}^\vee}

\nc{\Mod}{{\textrm{Mod}}}

\nc{\wGL}{{\widehat{GL}^+}}

\nc{\BA}{{\mathbb{A}}}
\nc{\BC}{{\mathbb{C}}}
\nc{\BM}{{\mathbb{M}}}
\nc{\BN}{{\mathbb{N}}}
\nc{\BF}{{\mathbb{F}}}
\nc{\BP}{{\mathbb{P}}}
\nc{\BR}{{\mathbb{R}}}
\nc{\BZ}{{\mathbb{Z}}}

\nc{\kk}{{\mathbb{K}}}

\nc{\CA}{{\mathcal{A}}}
\nc{\CB}{{\mathcal{B}}}
\nc{\CC}{{\mathcal{C}}}
\nc{\DD}{{\mathcal{D}}}
\nc{\CE}{{\mathcal{E}}}
\nc{\CF}{{\mathcal{F}}}
\nc{\tCF}{{\widetilde{\CF}}}
\nc{\oCF}{{\overline{\CF}}}
\nc{\CG}{{\mathcal{G}}}
\nc{\CL}{{\mathcal{L}}}
\nc{\CK}{{\mathcal{K}}}
\nc{\CI}{{\mathcal{I}}}
\nc{\CM}{{\mathcal{M}}}
\nc{\CH}{{\mathcal{H}}}
\nc{\CN}{{\mathcal{N}}}
\nc{\CO}{{\mathcal{O}}}
\nc{\CP}{{\mathcal{P}}}
\nc{\CR}{{\mathcal{R}}}
\nc{\CQ}{{\mathcal{Q}}}
\nc{\CS}{{\mathcal{S}}}
\nc{\CT}{{\mathcal{T}}}
\nc{\CU}{{\mathcal{U}}}
\nc{\CV}{{\mathcal{V}}}
\nc{\CW}{{\mathcal{W}}}

\nc{\CX}{{\mathcal{X}}}
\nc{\tCX}{{\widetilde{\mathcal{X}}}}
\nc{\CY}{{\mathcal{Y}}}
\nc{\tCY}{{\widetilde{\mathcal{Y}}}}

\nc{\tN}{{\widetilde{\CN}}}
\nc{\pN}{{\BP\widetilde{\CN}}}

\nc{\tT}{{T}}

\nc{\fC}{{\mathfrak{C}}}
\nc{\fZ}{{\mathfrak{Z}}}
\nc{\fU}{{\mathfrak{U}}}
\nc{\fS}{{\mathfrak{S}}}

\nc{\od}{{\overline{d}}}
\nc{\rg}{{\textrm{R}\Gamma}}
\nc{\erg}{{\emph{R}\Gamma}}
\nc{\id}{{\textrm{Id}}}
\nc{\rhom}{{\textrm{RHom}}}

\def\ph{\varphi}

\def\e{\varepsilon}

\def\and{\textrm{ }\&\textrm{ }}

\def\sym{\textrm{Sym}}

\def\tCF{\widetilde{\CF}}

\def\oCA{\tilde{\CA}}

\def\slz{SL_2(\BZ)}

\begin{document}

\title[The Shuffle Algebra Revisited]{The Shuffle Algebra Revisited}

\author[Andrei Negut]{Andrei Negut}
\address{Columbia University, Department of Mathematics, New York, NY, USA}
\address{Simion Stoilow Institute of Mathematics, Bucharest, Romania}
\email{andrei.negut@@gmail.com}

\maketitle

\begin{abstract}

In this paper we introduce certain new features of the shuffle algebra of \cite{F} that will allow us to obtain explicit formulas for the isomorphism between its Drinfeld double and the elliptic Hall algebra of \cite{BS}, \cite{SV}. These results are necessary for our work in \cite{GN} and \cite{Ne}, where they will be applied to the study of the Hilbert scheme and to computing knot invariants. \\

\end{abstract}

\section{Introduction}

The shuffle algebra $\CA^+$, first introduced by Feigin and Odesskii, is a subset of symmetric rational functions over the field $\BC(q_1,q_2)$, endowed with the shuffle product of \eqref{eqn:mult}. The elliptic Hall algebra $\CE$ was introduced by Burban and Schiffmann in \cite{BS} as the Hall algebra of the category of coherent sheaves on an elliptic curve. In \cite{SV}, Schiffmann and Vasserot have constructed an isomorphism $\Upsilon$ between the positive half of the elliptic Hall algebra and the shuffle algebra $\CA^+$. This isomorphism is given by generators and relations, and it extends to the Drinfeld doubles of the algebras in question. Our goal in this paper is to make this isomorphism $\Upsilon$ more explicit, by proving: \\

\begin{theorem} 
\label{thm:main}

We have:
\begin{equation}
\label{eqn:iso}
\Upsilon(u_{k,d}) = P_{k,d}
\end{equation}
where the $u_{k,d}$ are the standard generators of $\CE$ (see Subsection \ref{sub:elliptic} for the definition) and the $P_{k,d}$ are the minimal shuffle elements of Remark \eqref{rem:minimal}. As symmetric rational functions, they are represented by:
\begin{equation}
\label{eqn:70}
P_{k,d} = \frac {(q_1-1)^k(1-q_2)^k}{(q_1^{n}-1)(1-q_2^n)} \cdot
\end{equation}
$$
\emph{Sym} \left[ \frac {\prod_{i=1}^k z_i^{\left \lfloor \frac {id}k \right \rfloor - \left \lfloor \frac {(i-1)d}k \right \rfloor}\sum_{x=0}^{n-1} (q_1q_2)^{x} \frac {z_{a(n-1)+1}...z_{a(n-x)+1}}{{z_{a(n-1)}...z_{a(n-x)}}}}{\left(1 - q_1q_2 \frac {z_2}{z_1}\right)...\left(1 - q_1q_2 \frac {z_{k}}{z_{k-1}}\right)} \prod_{1\leq i<j \leq k} \omega \left( \frac {z_i}{z_j} \right) \right]
$$
where $n = \gcd(k,d)$, $a=\frac kn$, and $\omega(x) = \frac {(x-1)(x-q_1q_2)}{(x-q_1)(x-q_2)}$. \\

\end{theorem}

The above formula will feature in \cite{Ne}, where we will use it to identify $P_{k,d}$ with certain geometric operators that act on the $K-$theory of the moduli space of sheaves. By appealing to a large framework that connects knot invariants with Hilbert schemes through the category of representations of the rational Cherednik algebra, this will allow us to produce new formulas for torus knot invariants in \cite{GN}. Let us say a few things about the structure of this paper: \\

\begin{itemize}

\item In Section \ref{sec:shuffle}, we recall the definition of the shuffle algebra $\CA^+$, and introduce the crucial notion of slope. \\ 

\item In Section \ref{sec:elliptic}, we define the elliptic Hall algebra $\CE^+$ via generators and relations, and prove that it is isomorphic to $\CA^+$. \footnote{With our definition of the shuffle algebra, the map $\Upsilon:\CE^+ \longrightarrow \CA^+$ constructed by Schiffmann and Vasserot is a priori only known to be injective, and we will show that it is also surjective} \\

\item In Section \ref{sec:double}, we give formulas for the bialgebra structure on $\CA^+$, and show that $\Upsilon:\CE^+ \cong \CA^+$ preserves this structure, thus inducing an isomorphism $\CE \cong \CA$ between the Drinfeld doubles of the algebras in question. \\

\item In Section \ref{sec:imp}, we prove that $P_{k,d} := \Upsilon(u_{k,d})$ is completely determined by the minimality property of Remark \ref{rem:minimal}. \\

\item In Section \ref{sec:exp}, we show that the shuffle element in the RHS of \eqref{eqn:70} is minimal, thus proving Theorem \ref{thm:main} up to a constant. We then check that this constant is $1$. \\

\item In Section \ref{sec:app}, we present an Appendix where we prove the more computational results of the paper. \\

\end{itemize}

I would like to thank Igor Burban, Boris Feigin, Andrei Okounkov, Alexander Tsymbaliuk and Eric Vasserot for their interest and numerous helpful discussions. I am very grateful to the referee for many useful suggestions. \\

\section{The Shuffle Algebra}
\label{sec:shuffle}

\subsection{}

We will work over the field $\kk = \BC(q_1,q_2)$ and let us write $q=q_1q_2$. Consider an infinite set of variables $z_1,z_2,...$, and let us look at the $\kk-$vector space:
\begin{equation}
\label{eqn:big}
\CV = \bigoplus_{k \geq 0} \sym_{\kk}(z_1,...,z_k),
\end{equation}
bigraded by $k$ and homogenous degree. We endow it with a $\kk-$algebra structure via the \textbf{shuffle product}:
$$
P(z_1,...,z_k) * Q(z_1,...,z_l) =
$$
\begin{equation}
\label{eqn:mult}
= \frac 1{k!l!} \textrm{Sym} \left[P(z_1,...,z_k)Q(z_{k+1},...,z_{k+l}) \prod_{i=1}^k \prod_{j=k+1}^{k+l} \omega \left( \frac {z_i}{z_j} \right) \right]
\end{equation}
where:
\begin{equation}
\label{eqn:factor}
\omega(x)  = \frac {(x - 1)(x - q)}{(x - q_1)(x - q_2)}
\end{equation}
and \textrm{Sym} denotes the symmetrization operator on rational functions:
$$
\textrm{Sym}\left( P(z_1,...,z_k) \right) = \sum_{\sigma \in S(k)} P(z_{\sigma(1)},...,z_{\sigma(k)})
$$


\subsection{}

The \textbf{shuffle algebra} $\CA^+$ (see \cite{F}) \footnote{Actually, the algebra $\CA_0^+$ studied in \emph{loc. cit.} is defined with respect to the function $\omega_0(x,y)  = (x - y q^{-1}_1)(x - y q^{-1}_2)(x - yq)/(x-y)^3$ instead of \eqref{eqn:factor}. The two are isomorphic via the map: $$\CA^+ \longrightarrow \CA_0^+, \qquad \qquad P(z_1,...,z_k) \longrightarrow P(z_1,...,z_k) \prod_{1\leq i \neq j \leq k} \frac {(z_i-q_1z_j)(z_i-q_2z_j)}{q^{1/2}(z_i-z_j)^2}$$} is defined as the subspace of $\CV$ consisting of rational functions of the form:
\begin{equation}
\label{eqn:shuf}
P(z_1,...,z_k) =  \frac {p(z_1,...,z_k) \cdot \prod_{1\leq i<j \leq k} (z_i - z_j)^2}{\prod_{1\leq i \neq j \leq k} (z_i - q_1 z_j)(z_i - q_2 z_j)} , \qquad k \geq 1
\end{equation}
where $p$ is a symmetric Laurent polynomial that satisfies the \textbf{wheel conditions}:
\begin{equation}
\label{eqn:wheel}
p(z_1,z_2,z_3,...) = 0 \textrm{ whenever } \left\{ \frac {z_1}{z_2} , \frac {z_2}{z_3}, \frac {z_3}{z_1} \right\} =\left \{q_1,q_2,\frac 1q \right\}
\end{equation}
This condition is vacuous for $k\leq 2$. We will call elements of $\CA^+$ shuffle elements. The following proposition says that $\CA^+$ is an algebra: \\

\begin{proposition}
\label{prop:0}

If $P,P' \in \CA^+$, then $P*P' \in \CA^+$.  \\

\end{proposition}

\begin{proof} Let us write the shuffle elements $P(z_1,...,z_k)$ and $P'(z_1,...,z_{k'})$ in the form \eqref{eqn:shuf}. By \eqref{eqn:mult}, we have:
$$
P*P' = \frac 1{k! k'!} \cdot \frac {\prod_{1\leq i<j \leq k+k'} (z_i - z_j)^2}{\prod_{1\leq i \neq j \leq k+k'} (z_i - q_1 z_j)(z_i - q_2 z_j)} \cdot
$$
$$
\sym\left[ p(z_1,...,z_k) p'(z_{k+1},...,z_{k+k'}) \prod_{1\leq i\leq k < j \leq k+k'} \frac {(z_i- qz_j)(z_j-q_1z_i)(z_j-q_2z_i)}{z_i - z_j} \right] 
$$
The expression on the last line above is a rational function with at most simple poles at $z_j=z_i$. Because it is symmetric, it must necessarily be regular at $z_j=z_i$, and therefore it is a Laurent polynomial in the $z$ variables. To prove that $P*P'\in \CA^+$, we need only prove that the $\sym$ satisfies the wheel conditions \eqref{eqn:wheel}. In fact, we will show that every one of its summands does so. \\

To see this, note that we need to specialize three of the variables as in \eqref{eqn:wheel} and show that the given summand of the above $\sym$ vanishes. If all three of the chosen variables are among $\{z_1,...,z_k\}$ or $\{z_{k+1},...,z_{k+k'}\}$, then the summand vanishes because $p$ and $p'$ satisfy the wheel conditions themselves. If one of the variables is in $\{z_1,...,z_k\}$ and two are in $\{z_{k+1},...,z_{k+k'}\}$ (or vice-versa) then the product:
$$
\prod_{1\leq i\leq k < j \leq k+k'} \frac {(z_i- qz_j)(z_j-q_1z_i)(z_j-q_2z_i)}{z_i - z_j}
$$
vanishes, and therefore so does the respective summand. 

\end{proof}

\subsection{} Note that this definition of the shuffle algebra differs slightly from the one in \cite{SV}, where the authors actually work with the subalgebra:
\begin{equation}
\label{eqn:sub}
\oCA^+ \subset \CA^+
\end{equation}
generated by the elements $z_1^d \in \CA^+$, as $d\in \BZ$. An important result of the present paper is that the two algebras actually coincide, namely: \\

\begin{theorem}
\label{thm:coincide}

The above inclusion is an equality:
$$
\oCA^+ = \CA^+
$$ 
In other words, the shuffle algebra is generated by degree one elements. \\

\end{theorem}

\subsection{} 
\label{sub:degree}

The shuffle algebra $\CA^+$ is bigraded by the number of variables $k$ and the homogenous degree $d$ of our rational functions:
$$
\CA^+ = \bigoplus_{k\geq 0 , d\in \BZ} \CA_{k,d}
$$
Given a shuffle element $P(z_1,...,z_k) \in \CA_{k,d}$ and a number $\mu \in \BR$, we consider the limits: 
\begin{equation}
\label{eqn:limit}
\lim_{\xi \rightarrow \infty} \frac {P(\xi z_1,...,\xi z_i,z_{i+1},...,z_k)}{\xi^{\mu i}}
\end{equation}
We let $\CA_{k,d}^\mu \subset \CA_{k,d}$ denote the subspace of shuffle elements $P$ such that the above limits exist and are finite for all $i\in \{0,...,k\}$. Such a shuffle element $P$ is said to have \textbf{slope} $\leq \mu$. Then let us make the following simple observation, which will be given a proper proof in the Appendix: \\

\begin{proposition} 
\label{prop:algebra}

For any $\mu\in \BR$, the subspace:
$$
\CA^\mu:=\bigoplus_{k \geq 0, d\in \BZ} \CA^\mu_{k,d} \subset \CA^+
$$
is a subalgebra. \\

\end{proposition}

\subsection{} The subspaces $\CA^\mu_{k,d}$ give an increasing filtration of the infinite-dimensional vector space $\CA_{k,d}$:
$$
\CA_{k,d}^\mu \subset \CA_{k,d}^{\mu'} \qquad \textrm{if }\mu \leq \mu', \qquad \qquad \bigcup_{\mu} \CA_{k,d}^\mu = \CA_{k,d}
$$
The following proposition will show that the subspaces $\CA_{k,d}^\mu$ are finite dimensional, and it places an upper bound on their dimension. All of these bounds will be shown to be precise in Proposition \ref{prop:surj} below. \\ 


\begin{proposition} 
\label{prop:bound}

The vector space $\CA_{k,d}^\mu$ has dimension $\leq$ the number of unordered collections:
\begin{equation}
\label{eqn:steve}
(k_1,d_1),...,(k_t,d_t) \qquad  \textrm{such that} \quad \begin{cases} k_1+...+k_t = k, \\ d_1+...+d_t = d, \\ d_i \leq \mu k_i \quad \forall i \end{cases}
\end{equation}
where $t\geq 1$ is any natural number, $k_i\in \BN$ and $d_i \in \BZ$. \\

\end{proposition}

\begin{proof} This proposition was first stated and proved in the special case $d=\mu=0$ in \cite{F}, and we will generalize their idea in order to obtain the desired result. For any partition $\rho = \{k_1 \geq... \geq k_t>0\}$ of $k$, consider the evaluation map:
$$
\ph_{\rho}:\CA^\mu_{k,d} \longrightarrow \kk[y^{\pm 1}_1,...,y^{\pm 1}_{t}], 
$$

$$
\ph_\rho(P) = p(z_1,...,z_k)|_{z_{k_1+...+k_{s-1}+x}=y_s q^x, \quad \forall s \in \{1,...,t\} \ \forall x\in \{1,...,k_s\}}
$$
where $p$ is the Laurent polynomial of \eqref{eqn:shuf}. This construction gives rise to the subspaces:
\begin{equation}
\label{eqn:filt}
\CA^\mu_{k,d} \supset \CA^{\mu,\rho}_{k,d} = \bigcap_{\rho' > \rho} \ker \ph_{\rho'} 
\end{equation}
where $>$ is the dominance ordering \footnote{We recall that the dominance ordering is $\rho'\geq \rho \Leftrightarrow \rho'_1+...+\rho'_i \geq \rho_1+...+\rho_i$ for all $i$. We write $\rho'>\rho$ if $\rho'\geq \rho$ and $\rho'\neq \rho$} on partitions of $k$. It is easy to see that these subspaces form a filtration of $\CA^\mu_{k,d}$ (if we set $\CA^{\mu,(k)}_{k,d}=\CA^\mu_{k,d}$), namely:
$$
\rho \leq \rho' \Longrightarrow \CA^{\mu,\rho}_{k,d} \subset \CA^{\mu,\rho'}_{k,d}
$$
Then the desired upper bound on the dimension of $\CA_{k,d}^\mu$ would follow from the inequalities:
\begin{equation}
\label{eqn:ineqclaim}
\dim \ph_\rho \left(\CA^{\mu,\rho}_{k,d} \right) \leq \ \# (d_1,...,d_t) \ \textrm{ such that \eqref{eqn:steve} holds}
\end{equation}
It should be remarked that the RHS counts the number of partially ordered tuples $(d_1,...,d_t)$. This means that if $k_i=k_j$ then we disregard the ordering between $d_i$ and $d_j$. Let us now prove \eqref{eqn:ineqclaim}. Take a shuffle element $P \in \CA^{\mu,\rho}_{k,d}$ and look at the Laurent polynomial $r=\ph_\rho(P)$. This Laurent polynomial is partially symmetric, in the same sense as the RHS of \eqref{eqn:ineqclaim} is partially ordered: if $k_i=k_j$ then $r$ is symmetric in $y_i$ and $y_j$. Because $P$ satisfies the wheel conditions \eqref{eqn:wheel}, the Laurent polynomial $r$ vanishes for:
\begin{equation}
\label{eqn:zero1}
y_j = q_2 q^{a-b} y_i, \quad a \in \{1,...,k_i-1\}, \quad b \in \{1,...,k_j \}
\end{equation}
	
\begin{equation}
\label{eqn:zero2}
y_j = q_1 q^{a-b} y_i, \quad a \in \{1,...,k_i-1\}, \quad b \in \{ 1,...,k_j \}
\end{equation}
for $i<j$, with the correct multiplicities. Because $P$ lies in $\CA^{\mu,\rho}_{k,d} = \bigcap_{\rho' > \rho} \ker \ph_\rho$, we see that $r$ also vanishes for:
\begin{equation}
\label{eqn:zero3}
y_j = q^{k_i-b+1}y_i \quad \textrm{ and } \quad y_j = q^{-b}y_i, \qquad b \in \{1,...,k_j\}
\end{equation}
for $i<j$. Therefore, the Laurent polynomial $r$ is divisible by:
$$
r_0 = \prod_{1\leq i < j \leq t} \left[ \prod_{b=1}^{k_j} (y_j-q^{k_i-b+1}y_i)(y_j-q^{-b}y_i) \prod_{b=1}^{k_j} \prod_{a=1}^{k_i-1} (y_j - q_2 q^{a-b}y_i)(y_j - q_1 q^{a-b}y_i) \right]
$$
This polynomial has total degree:
$$
\deg(r_0) = \sum_{i<j} 2k_i k_j = k^2 - \sum_i k_i^2
$$
and degree at most:
$$
\deg_{y_i}(r_0) = \sum_{i\neq j} 2k_i k_j = 2k k_i - 2k_i^2
$$
in each variable $y_i$. As for $r$, it's easy to see that it has total degree:
$$
\deg(r) = k(k-1)+d
$$
Because the slope of $P$ is $\leq \mu$, then it has degree in each variable at most:
$$
\deg_{y_i}(r) \leq 2k k_i - k_i(k_i+1) + \mu k_i
$$
So the quotient $r/r_0$ is a Laurent polynomial of total degree:
$$
\deg(r/r_0) = \sum_i k_i(k_i-1) + d
$$
and degree in each variable at most:
$$
\deg_{y_i}(r/r_0) \leq k_i(k_i-1) + \mu k_i
$$
Such Laurent polynomials are spanned by monomials:
$$
y_1^{d_1+k_1(k_1-1)} ... y_t^{d_t+k_t(k_t-1)}
$$
where $d_1+...+d_t = d$ and $d_i \leq \mu k_i$ for all $i$. When $k_i=k_j$, both $r$ and $r_0$ are symmetric in $y_i$ and $y_j$, so we disregard the order between $d_i$ and $d_j$ in the above count. We conclude that $r=\ph_\rho(P)$ lies in a vector space of dimension exactly equal to the RHS of \eqref{eqn:ineqclaim}, thus completing the proof. 

\end{proof}

\begin{remark}
\label{rem:minimal}

The same proof also shows that the subspace of $\CA_{k,d}$ consisting of shuffle elements such that:
$$
\lim_{\xi \rightarrow \infty} \frac {P(\xi z_1,...,\xi z_i,z_{i+1},...,z_k)}{\xi^{\frac {di}k}} = 0 \qquad \forall \quad i\in \{1,...,k-1\}
$$
is at most one-dimensional. We will show that it is exactly one-dimensional, and will therefore call shuffle elements $P$ that verify this condition \textbf{minimal}. \\

\end{remark}

\section{The Elliptic Hall algebra}
\label{sec:elliptic}

\subsection{}
\label{sub:semi}

In this section, we will often work with \textbf{quasi-empty} triangles, which we define by the condition that their vertices are of the form  $X=(0,0)$, $Y=(k_2,d_2)$, $Z=(k_1+k_2,d_1+d_2)$, and satisfy: \\

\begin{itemize} 

\item $k_1,k_2 > 0$ \\

\item $\frac {d_1}{k_1}> \frac {d_2}{k_2}$ \\

\item there are no lattice points inside the triangle, nor on at least one of the edges $XY$, $YZ$ \\

\end{itemize} 

If there are no points on both $XY$ and $YZ$, we call the triangle \textbf{empty}. For example, the triangle below is empty. \\

\begin{picture}(100,150)(-110,-15)
\label{pic:par}

\put(0,0){\circle*{2}}\put(20,0){\circle*{2}}\put(40,0){\circle*{2}}\put(60,0){\circle*{2}}\put(80,0){\circle*{2}}\put(100,0){\circle*{2}}\put(120,0){\circle*{2}}\put(0,20){\circle*{2}}\put(20,20){\circle*{2}}\put(40,20){\circle*{2}}\put(60,20){\circle*{2}}\put(80,20){\circle*{2}}\put(100,20){\circle*{2}}\put(120,20){\circle*{2}}\put(0,40){\circle*{2}}\put(20,40){\circle*{2}}\put(40,40){\circle*{2}}\put(60,40){\circle*{2}}\put(80,40){\circle*{2}}\put(100,40){\circle*{2}}\put(120,40){\circle*{2}}\put(0,60){\circle*{2}}\put(20,60){\circle*{2}}\put(40,60){\circle*{2}}\put(60,60){\circle*{2}}\put(80,60){\circle*{2}}\put(100,60){\circle*{2}}\put(120,60){\circle*{2}}\put(0,80){\circle*{2}}\put(20,80){\circle*{2}}\put(40,80){\circle*{2}}\put(60,80){\circle*{2}}\put(80,80){\circle*{2}}\put(100,80){\circle*{2}}\put(120,80){\circle*{2}}\put(0,100){\circle*{2}}\put(20,100){\circle*{2}}\put(40,100){\circle*{2}}\put(60,100){\circle*{2}}\put(80,100){\circle*{2}}\put(100,100){\circle*{2}}\put(120,100){\circle*{2}}\put(0,120){\circle*{2}}\put(20,120){\circle*{2}}\put(40,120){\circle*{2}}\put(60,120){\circle*{2}}\put(80,120){\circle*{2}}\put(100,120){\circle*{2}}\put(120,120){\circle*{2}}

\put(0,0){\line(2,1){42}}
\put(0,0){\line(3,2){120}}
\put(40,20){\line(4,3){80}}
\put(0,-10){\vector(0,1){150}}
\put(-10,0){\vector(1,0){150}}

\put(37,-8){\scriptsize{$k_2$}}
\put(77,-8){\scriptsize{$k_1$}}
\put(112,-8){\scriptsize{$k_1+k_2$}}

\put(-10,18){\scriptsize{$d_2$}}
\put(-10,58){\scriptsize{$d_1$}}
\put(-28,78){\scriptsize{$d_1+d_2$}}

\put(40,-20){\mbox{Figure \ref{pic:par}}}

\end{picture}

\subsection{}
\label{sub:elliptic}

The elliptic Hall algebra $\CE$ was studied in detail by Burban and Schiffmann in \cite{BS}, and we would like to compare the shuffle algebra with their viewpoint. By definition (\cite{BS},\cite{SV}), its positive half $\CE^+$ is generated by elements $u_{k,d}$ for $k \geq 1, d\in \mathbb{Z}$, under the relations:
\begin{equation}
\label{eqn:relation0}
[u_{k_1,d_1}, u_{k_2,d_2}] = 0, 
\end{equation}
whenever the points $(k_1,d_1), (k_2,d_2)$ are collinear, and:
\begin{equation}
\label{eqn:relation}
[u_{k_1,d_1}, u_{k_2,d_2}] = \frac {\theta_{k_1+k_2,d_1+d_2}}{\alpha_1}
\end{equation}
whenever the triangle with vertices $(0,0),(k_2,d_2),(k_1+k_2,d_1+d_2)$ is quasi-empty in the sense of Section \ref{sub:semi}. Here we set: 
\begin{equation}
\label{eqn:alpha}
\alpha_n = \frac {(q_1^n-1)(q_2^n-1)(q^{-n}-1)}n
\end{equation}

\begin{equation}
\label{eqn:exponential} 
\qquad \sum_{n=0}^{\infty} x^n \theta_{na,nb}  = \exp \left( \sum_{n=1}^\infty  \alpha_n  x^n u_{na,nb} \right)
\end{equation} 
for any $\gcd(a,b)=1$. The algebra $\CE^+$ is also bigraded by the two coordinates $k$ and $d$. The following Theorem has been proved in \cite{SV}:\\

\begin{theorem} 
\label{thm:sv}

The map $u_{1,d} \longrightarrow z_1^d$ extends to an injective algebra morphism \footnote{A priori, our definition of the shuffle algebra is larger than the one used in \cite{SV}, and we therefore cannot infer that this map is also surjective. This will be proved in Proposition \ref{prop:surj} below}:

$$
\Upsilon: \CE^+ \longrightarrow \CA^+
$$
$$$$
\end{theorem}

\subsection{} From \eqref{eqn:relation}, it is clear that $\CE^+$ is generated by degree 1 elements $u_{1,d}$. Therefore, the image of the map $\Upsilon$ is precisely the subalgebra $\oCA^+$ of \eqref{eqn:sub}. This means that Theorem \ref{thm:coincide} follows from the following: \\ 

\begin{proposition}
\label{prop:surj} 

The map $\Upsilon$ of Theorem \ref{thm:sv} is surjective. \\

\end{proposition}

\begin{proof} Let $\CE_{k,d} \subset \CE^+$ denote the subspace of elements of bidegrees $(k,d)$, and for any slope $\mu \in \BR$ consider:
$$
\CE^\mu_{k,d} = \left\{\textrm{sums of products of }u_{k',d'}\textrm{ for } \frac {d'}{k'} \leq \mu \right\} \subset \CE_{k,d}
$$
By Lemma 5.6 of \cite{BS}, the dimension of the above vector space precisely equals the number of tuples as in \eqref{eqn:steve}. Indeed, this comes about because this dimension equals the number of convex paths in $\textbf{Conv}^+$ of slope $\leq \mu$ (in the notation of \emph{loc. cit.}), and such paths are in 1-1 correspondence with unordered collections $\{(k_1,d_1),...,(k_t,d_t)\}$ of slope $\leq \mu$ which sum up to $(k,d)$. By Proposition \ref{prop:bound}, the finite-dimensional subspace $\CA^\mu_{k,d}$ has dimension at most equal to the same number. Therefore, the desired surjectivity would follow from the claim:
$$
\Upsilon(\CE^\mu_{k,d}) \subset \CA^\mu_{k,d}
$$
(the injectivity of $\Upsilon$ is contained in Theorem \ref{thm:sv}). By Proposition \ref{prop:algebra}, it is enough to show that:
\begin{equation}
\label{eqn:defp}
P_{k,d} := \Upsilon(u_{k,d})
\end{equation}
has slope $\leq \frac dk$. We will prove this statement by induction on $k$. It is trivial for $k=1$, because $P_{1,d} = z_1^d$ by the definition of $\Upsilon$. Assume the claim true for all $k'<k$, and let us prove it for $P_{k,d}$. Take an empty triangle with vertices $(0,0), (k_2,d_2),(k,d)$, which exists simply by choosing one of minimal area. Relation \eqref{eqn:relation} and the fact that $\Upsilon$ is a morphism imply that:
\begin{equation}
\label{eqn:smaug}
P_{k,d} = \frac {\alpha_1}{\alpha_{\gcd(k,d)}} [P_{k_1,d_1},P_{k_2,d_2}] + \left( \textrm{products of }P_{k'd'}\textrm{ with }\frac {k'}{d'} = \frac kd \right)
\end{equation}
The fact that the second summand has slope $\leq \frac dk$ follows from the induction hypothesis. Then our claim is equivalent to showing that the commutator $[P_{k_1,d_1},P_{k_2,d_2}]$ has slope $\leq \frac dk$. The induction hypothesis tells us that $P_{k_1,d_1}$ and $P_{k_2,d_2}$ have slopes $\leq \frac {d_1}{k_1}$ and $\frac {d_2}{k_2}$, respectively, and therefore by Proposition \ref{prop:algebra} their product has slope: 
$$
\leq \max \left( \frac {d_1}{k_1},\frac {d_2}{k_2} \right) = \frac {d_1}{k_1}
$$ 
Since $\frac dk < \frac {d_1}{k_1}$, the above estimate is not good enough. We will finesse the inequality by using \eqref{eqn:generalclaim}. The shuffle elements $P_{k_1,d_1}*P_{k_2,d_2}$ and $P_{k_2,d_2} * P_{k_1,d_1}$ are rational functions in $k=k_1+k_2$ variables. As we multiply any $j\leq k$ of these variables by $\xi \longrightarrow \infty$, the resulting term is of order:
$$
\max_{i+i'=j} \left \lfloor \frac {d_1i}{k_1} \right \rfloor +  \left \lfloor \frac {d_2i'}{k_2} \right \rfloor 
$$
in $\xi$. If this expression were $\leq   \frac {dj}k  $, we would be done with proving that the commutator of \eqref{eqn:smaug} has slope $\leq \frac dk$. Since there are no lattice points inside the quasi-empty triangle, the only case when this inequality fails to hold is when $i=k_1$ and $i'=0$. By \eqref{eqn:generalclaim}, the term which arises in this way is:
$$
\xi^{d_1} P_{k_1,d_1}(z_1,...,z_{k_1}) \cdot P_{k_2,d_2}(z_{k_1+1},...,z_{k_1+k_2}) + O\left(\xi^{\frac {k_1d}k} \right)
$$
However, the leading order term above appears in both $P_{k_1,d_1}*P_{k_2,d_2}$ and $P_{k_2,d_2} * P_{k_1,d_1}$, and therefore drops out in their commutator. We conclude that this term does not appear in \eqref{eqn:smaug}, hence $P_{k,d}$ has slope $\leq \frac dk$. 

\end{proof}

\section{The Double Shuffle Algebra}
\label{sec:double}

\subsection{} Given a bialgebra \footnote{In all our bialgebras, the coproduct $\Delta$ is coassociative and compatible with the product $*$, in the sense that $\Delta(a*b) = \Delta(a)*\Delta(b)$} $A$, a symmetric non-degenerate pairing:
$$
(\cdot,\cdot):A \otimes A \longrightarrow \BC
$$
such that:
\begin{equation}
\label{eqn:bialg}
(a*b,c) = (a\otimes b,\Delta(c)) \qquad \forall a,b,c \in A
\end{equation}
is called a \textbf{bialgebra pairing}. To such a datum, one can associate the Drinfeld double of the bialgebra $A$ (see, for example \cite{D}). To define it, recall the Sweedler notation for the coproduct:
$$
\Delta(a) = a_{1} \otimes a_{2}
$$
where the RHS implicitly contains a sum over several terms. The \textbf{Drinfeld double} of the bialgebra $A$ with the pairing \eqref{eqn:bialg} is $\DD A = A^{\textrm{coop}} \otimes A$ as a vector space, with the property that $A^- = A^{\textrm{coop}} \otimes 1$ and $A^+ = 1 \otimes A$ are both sub-bialgebras of $\DD A$, and we impose the extra relation:
$$
a_{1}^- * b_{2}^+ \cdot (a_{2},b_{1}) = b_{1}^+ * a_{2}^- \cdot (b_{2},a_{1}) \qquad \forall a,b \in A
$$
where $a^- = a \otimes 1$ and $b^+ = 1 \otimes b$. This latter condition teaches us how to commute elements from the two factors of the Drinfeld double, and it uniquely determines the bialgebra structure on $\DD A$. \\

\subsection{}

There is no coproduct of interest on the shuffle algebra $\CA^+$, but we will find one on a slightly larger algebra. Let $\CA^\geq$ be generated by $\CA^+$ and commuting elements $h_0,h_1,...$ under the relation:
\begin{equation}
\label{eqn:relhp}
P(z_1,...,z_k) * h(w) = h(w) * \left[ P(z_1,...,z_k) \prod_{i=1}^k \Omega \left( \frac w{z_i} \right) \right]
\end{equation}
where $h(w) = \sum_{n \geq 0} h_nw^{-n}$ and: 
\begin{equation}
\label{eqn:defomega}
\Omega(x) := \frac {\omega(1/x)}{\omega(x)}=\frac {(x-q^{-1})(x-q_1)(x-q_2)}{(x-q)(x-q_1^{-1})(x-q_2^{-1})} = \exp \left(- \sum_{n\geq 1} \alpha_n x^{-n} \right)
\end{equation}
One makes sense of relation \eqref{eqn:relhp} by expanding the RHS in negative powers of $w$. The reason for introducing these new generators is to define the coproduct. \\

\begin{proposition}
\label{prop:copcheck} 

The following assignments give rise to a coproduct on $\CA^\geq$:
$$
\Delta(h(w)) = h(w) \otimes h(w),
$$
\begin{equation}
\label{eqn:coproduct}
\Delta(P(z_1,...,z_k)) = \sum_{i=0}^k \frac {\prod_{b>i} h(z_{b}) \cdot P(z_1,...,z_i \otimes z_{i+1},...,z_k)}{\prod_{a\leq i < b} \omega(z_b/z_a)}
\end{equation}

\end{proposition}

The meaning of the above RHS is that we expand the fraction in non-negative powers of $z_a/z_b$ for $a\leq i < b$, thus obtaining an infinite sum of monomials. In each of those monomials, we put all the $h_n$'s to the very left of the expression, then all powers of $z_1,...,z_i$ to the left of the $\otimes$ sign, and finally all powers of $z_{i+1},...,z_k$ to the right of the $\otimes$ sign, as in the following example:
$$
\Delta(P) = ... + h_{n_{i+1}}...h_{n_k} \cdot z_1^{c_1}...z_i^{c_i} \otimes z_{i+1}^{c_{i+1}}...z_k^{c_k} + ...
$$
We obtain an expression which is a power series in $z_a$ for $a\leq i$ and in $z_b^{-1}$ for $b>i$, so the above tensor product takes values in a completion of $\CA^\geq \otimes \CA^\geq$. Proposition \ref{prop:copcheck} will be proved in the Appendix. \\

\subsection{}

The bialgebra $\CA^\geq$ has a pairing, defined by:
\begin{equation}
\label{eqn:pair0}
(h(v),h(w^{-1})) = \Omega \left( \frac {w}{v} \right)
\end{equation}

\begin{equation}
\label{eqn:pair}
\left( P , P' \right) = \frac {1} {\alpha_1^k} :\int: \frac {P(u_1,...,u_k) P'\left(\frac 1{u_1},...,\frac 1{u_k} \right)}{\prod_{1\leq i \neq j \leq k} \omega(u_i / u_j)} Du_1...Du_k \qquad  
\end{equation}
for all $P,P'\in \CA_{k,d}$. In the above, we set $Du = \frac {du}{2\pi i u}$ and $:\int:$ denotes the \textbf{normal-ordered integral}. We define it by:
$$
\left(\sym \left[ z_1^{n_1}...z_k^{n_k} \prod_{1\leq i < j \leq k} \omega(z_i/z_j) \right], P \right) =
$$

\begin{equation}
\label{eqn:normal}
= \frac {1} {\alpha_1^k} \int_{|u_1| \ll |u_2| \ll ... \ll |u_k|} \frac {u_1^{n_1}...u_k^{n_k}  P\left(\frac 1{u_1},...,\frac 1{u_k} \right)}{\prod_{i>j} \omega(u_i/u_j)} Du_1...Du_k
\end{equation}
for all $n_1,...,n_k \in \BZ$ such that $d=n_1+...+n_k$. By Proposition \ref{prop:surj}, this is enough to define the pairing on the whole shuffle algebra $\CA^+$, since any shuffle element equals a linear combination of products $z^{n_1} * ... * z^{n_k}$. However, there might exist linear relations between these products, so we need to check that the pairing is unambiguously defined. The following Proposition will be proved in the Appendix: \\

\begin{proposition}
\label{prop:well}

The above induces a well-defined bialgebra pairing: 
$$\CA^\geq  \otimes \CA^\geq \longrightarrow \kk,$$
in the sense of \eqref{eqn:bialg}. \\

\end{proposition}

\subsection{} We let $\CA = \DD \CA^\geq$ be the Drinfeld double of the shuffle algebra with respect to the above bialgebra pairing, and call $\CA$ the \textbf{double shuffle algebra}. On the other hand, we consider the algebra $\CE^\geq$ generated by $\CE^+$, a central element $c$, \footnote{Note that the abstract algebra studied in \cite{BS} has one more central element $c'$, but we will not need it in our applications and so we set it equal to 1. The results of the present paper hold just as well were $c'$ added, but some formulas would be messier} and commuting elements $u_{0,1},u_{0,2},...$, under the relations:
\begin{equation}
\label{eqn:relsasha}
[u_{0,d},u_{1,d'}] = u_{1,d+d'} \qquad \forall d \in \BZ, d'>0
\end{equation}
This is a bialgebra, with coproduct given by:
$$
\Delta(u_{0,d}) = u_{0,d} \otimes 1 + 1 \otimes u_{0,d}, \qquad \Delta(u_{1,d}) = u_{1,d} \otimes 1 + c \sum_{n\geq 0} \theta_{0,n} \otimes u_{1,d-n}
$$
for all $d \in \BZ$, where the $\theta_{0,n}$ are obtained from the $u_{0,n}$ according to \eqref{eqn:exponential}. There is a bialgebra pairing on $\CE^\geq$ thus defined, given by:
\begin{equation}
\label{eqn:pairhall}
(u_{0,d},u_{0,d}) = \frac {1}{\alpha_d}, \qquad (u_{1,d},u_{1,d}) = \frac {1}{\alpha_1}
\end{equation}
The Drinfeld double $\CE = \DD \CE^\geq$ is the elliptic Hall algebra studied in \cite{BS}, \cite{SV}. \\


\begin{theorem}
\label{thm:iso}

The injective morphism of Theorem \ref{thm:sv} extends to an isomorphism $\Upsilon:\CE^\geq \longrightarrow \CA^\geq$ via: 
$$
c \longrightarrow h_0 \qquad \text{ and } \qquad u_{0,d} \longrightarrow p_d, \quad \forall \ d>0
$$
where $p_1,p_2,...\in \CA^0$ are obtained from the series: 
$$
h(w)  = h_0 \cdot \exp \left( \sum_{n=1}^\infty  \alpha_n  p_n w^{-n} \right)
$$
This extended isomorphism $\Upsilon$ preserves the coproduct and the bialgebra pairing, and hence induces an isomorphism (denoted by the same letter):
$$
\Upsilon:\CE \longrightarrow \CA
$$ 
of their Drinfeld doubles. \\

\end{theorem}

\subsection{} Let us give a few details about the proof of the above Theorem, although proving it will be left as an exercise. At the level of $\CE^+ \hookrightarrow \CE^\geq$, the fact that $\Upsilon$ is an isomorphism follows from Theorem \ref{thm:sv} and Proposition \ref{prop:surj}. We need to upgrade this isomorphism to the whole of $\CE^\geq$, i.e. to check that it matches relation \eqref{eqn:relhp} with relation \eqref{eqn:relsasha}. This follows easily from the following equation, obtained by taking the logarithm of \eqref{eqn:relhp}:
\begin{equation}
\label{eqn:hecke}
[p_n , P(z_1,...,z_k)] = P(z_1,...,z_k)(z_1^n+...+z_k^n) \quad \forall P \in \CA^+
\end{equation}
In particular, commuting shuffle elements with the $p_n$ gives rise to the action of the ring of symmetric functions on $\CA^+$ by Hecke operators that was described in \cite{SV}. When $n=d$ and $k=1$, \eqref{eqn:hecke} matches with \eqref{eqn:relsasha}, hence $\Upsilon:\CE^\geq \longrightarrow \CA^\geq$ is an algebra morphism. It is straightforward to check that it preserves the coproduct and pairing, as it is enough to prove it at the level of the generators $u_{1,d} \longrightarrow z_1^d$. \\ 

\subsection{} 
\label{sub:slz}

The Drinfeld double $\CE$ is generated by symbols $u_{k,d}$ for $(k,d)\in \BZ^2 \backslash 0$, where: 
$$u_{k,d} = u_{k,d}^+, \qquad u_{-k,-d} = u_{k,d}^-, \qquad \forall k>0 \text{ or } k=0,d>0$$
If we set $h_0=1$, then as described in \cite{BS} any element:
$$
\gamma = \left( \begin{array}{cc}
x & y \\
z & w \end{array} \right) \in \slz
$$ 
gives rise to an automorphism:
$$
g_\gamma:\CE|_{h_0=1} \longrightarrow \CE|_{h_0=1}
$$
by permuting the generators $u_{k,d} \longrightarrow u_{xk+yd,zk+wd}$. These morphisms have the property that $g_\gamma\circ g_{\gamma'} = g_{\gamma\gamma'}$ and $g_1=\textrm{Id}$, so they give an action of $\slz$ on the algebra $\CE|_{h_0=1}$. If we restore the formal parameter $h_0$, then the universal cover of $\slz$ is the one that acts on $\CE$, as described in \cite{BS}. \\

\section{Describing $P_{k,d}$ implicitly}
\label{sec:imp}

\subsection{} 
\label{sub:b}

Recall the subalgebra $\CA^\mu \subset \CA^+$ that was introduced in Proposition \ref{prop:algebra}. From the definition of the coproduct $\Delta$, we infer the following: \\

\begin{proposition}
\label{prop:radu} 

For all $P\in \CA_{k,d}^\mu$, we have:
\begin{equation}
\label{eqn:radu}
\Delta(P) = \sum_{i=0}^{k} h_0^{k-i} \lim_{\xi \longrightarrow \infty} \frac {P( z_{j \leq i} \otimes \xi \cdot z_{j>i})}{\xi^{\mu (k-i)}} + (\emph{anything}) \otimes (\emph{slope} < \mu) \qquad 
\end{equation}
Recall that the tensor product inside the rational function $P$ means that all powers of $z_{j\leq i}$ go to the left of the $\otimes$ sign, while all powers of $z_{j>i}$ go to the right thereof. \\

\end{proposition}

In particular, it is easy to see that:
\begin{equation}
\label{eqn:radu2}
\Delta_\mu(P) = \sum_{i=0}^{k} h_0^{k-i} \lim_{\xi \longrightarrow \infty} \frac {P(z_{j \leq i} \otimes \xi \cdot  z_{j>i})}{\xi^{\mu (k-i)}}
\end{equation}
is a coproduct on the subalgebra:
\begin{equation}
\label{eqn:b}
\CB^\mu = \bigoplus_{k\geq 0}^{d=\mu k} \CA^\mu_{k,d} \subset \CA^\mu \subset \CA^+
\end{equation}
which should be interpreted as the leading term of $\Delta$ to order $\mu$. An easy consequence of Proposition \ref{prop:radu}, and the fact that the pairing $(\cdot,\cdot)$ preserves the grading, is the following: \\

\begin{proposition}

The pairing \eqref{eqn:pair} on $\CB^\mu$ satisfies the bialgebra property with respect to the shuffle product and the coproduct $\Delta_\mu$. \\

\end{proposition}

\subsection{} 

In \eqref{eqn:defp}, $P_{k,d}\in \CA^+$ has been defined as the image of the generator $u_{k,d} \in \CE^+$ under the isomorphism $\Upsilon$ of Theorem \ref{thm:iso}. Our main Theorem \ref{thm:main} requires us to give an explicit formula for this element. Before we do so, we will give an implicit description of $P_{k,d}$. \\

\begin{lemma} 
\label{lem:imp}

Up to a constant multiple, $P_{k,d}$ is the unique element of $\CA_{k,d}$ such that:
\begin{equation}
\label{eqn:gangnam}
\Delta(P_{k,d}) = P_{k,d} \otimes 1 + h_0^k \otimes P_{k,d} + (\emph{anything}) \otimes \left(\emph{slope} < \frac dk \right) \in \CA^\geq \otimes \CA^\geq
\end{equation}
Moreover, if $\gcd(k,d)=1$ and $(0,0)$, $(k_2,d_2)$, $(k,d)$ is a quasi-empty triangle in the sense of Subsection \ref{sub:semi}, then the component of $\Delta(P_{k,d})$ in $\CA_{k_1,d_1} \otimes \CA_{k_2,d_2}$ equals:
$$
\frac {h_0^{k_2}Q_{k_1,d_1} \otimes Q_{k_2,d_2}}{\alpha_1}
$$
where the $Q_{k,d}$ are computed from the $P_{k,d}$ by the following relation:
\begin{equation}
\label{eqn:shi}
1 + \sum_{n=1}^{\infty} x^n Q_{na,nb}  = \exp \left( \sum_{n=1}^\infty  \alpha_n  x^n P_{na,nb} \right), \qquad \textrm{whenever } \gcd(a,b)=1
\end{equation}

\end{lemma}

\begin{proof} Any shuffle element $P$ which satisfies \eqref{eqn:gangnam} has, by \eqref{eqn:radu}, the property that the limits:
$$
\lim_{\xi \longrightarrow \infty} \frac {P( \xi z_1,..., \xi z_i, z_{i+1},..., z_k)}{\xi^{ \mu i }}
$$
vanish for all $i\in \{1,...,k-1\}$. By Remark \eqref{rem:minimal}, the space of such shuffle elements is at most one-dimensional, hence the uniqueness of $P$. Now we need to show that $P_{k,d} = \Upsilon(u_{k,d})$ satisfies both the above conditions, which we will do by induction on $k$. The case $k=1$ is obvious, since $P_{1,d} = z_1^d$. Then let us assume the lemma is true for all $i<k$ and pick an empty triangle $(0,0)$, $(k_2,d_2)$, $(k,d)$. Such a triangle always exists, simply by picking one of minimal area. Since $\Upsilon$ preserves the relations inside $\CE$, we have the following equality in $\CA$:
\begin{equation}
\label{eqn:relationx}
Q_{k,d} = \alpha_1 [P_{k_1,d_1},P_{k_2,d_2}]
\end{equation}
where $(k_1,d_1) = (k,d) - (k_2,d_2)$. The induction hypothesis implies that:
$$
\Delta(P_{k_1,d_1}) = P_{k_1,d_1} \otimes 1 + h_0^{k_1} \otimes P_{k_1,d_1} + \sum_{dx=ky}^{0 < x < k_1} h_0^x P_{k_1-x,d_1-y} \otimes Q_{x,y} + ... 
$$

$$
\Delta(P_{k_2,d_2}) = P_{k_2,d_2} \otimes 1 + h_0^{k_2} \otimes P_{k_2,d_2} + \sum_{dx=ky}^{0<x<k_2} h_0^{k_2-x} Q_{x,y} \otimes P_{k_2-x,d_2-y} + ...
$$
The ellipsis denotes terms whose second tensor factor has smaller slope. Since $h_0$ is central, taking the commutator of these expressions implies:
$$
\Delta(Q_{k,d}) = Q_{k,d}\otimes 1 + h_0^k \otimes Q_{k,d} + \alpha_1 \sum_{dx=ky}^{0 < x < k_1}  h_0^x [P_{k_1-x,d_1-y},P_{k_2,d_2}] \otimes Q_{x,y} +
$$

$$
+ \alpha_1  \sum_{dx=ky}^{0 < x < k_2} h_0^{k-x} Q_{x,y} \otimes [P_{k_1,d_1},P_{k_2-x,d_2-y}]+ (\text{anything}) \otimes \left(\text{slope} < \frac dk \right)
$$
Applying relation \eqref{eqn:relation} gives us:
$$
\Delta(Q_{k,d}) = Q_{k,d}\otimes 1 + h_0^k \otimes Q_{k,d} + \sum_{dx=ky}^{0 < x < k_1}  h_0^xQ_{k-x,d-y} \otimes Q_{x,y} + \sum_{dx=ky}^{0<x<k_2} h_0^{k-x} Q_{x,y} \otimes Q_{k-x,d-y} 
$$

$$
+ (\text{anything}) \otimes \left(\text{slope} < \frac dk \right) = \sum_{dx=ky}^{0 \leq x \leq k}  h_0^{x} Q_{k-x,d-y} \otimes Q_{x,y} + (\text{anything}) \otimes \left(\text{slope} < \frac dk \right) 
$$
Let $(k,d)=(na,nb)$ for $\gcd(a,b)=1$. The above argument shows that $Q_{k,d}$ lies in the subalgebra $\CB^{\frac ba}$ of \eqref{eqn:b}. Therefore, we can analyse these elements with the coproduct $\Delta_{b/a}$, which is obtained from $\Delta$ by removing all terms of slope $<b/a = d/k$ in the second tensor factor. The above implies:
$$
\Delta_{b/a}(Q_{na,nb}) = \sum_{x=0}^n h_0^{xa} Q_{(n-x)a,(n-x)b} \otimes Q_{xa,xb}
$$
Elements whose coproduct satisfies the above property are called group-like, and it is well-known that they are exponents of primitive elements \footnote{The usual definition of group-like (respectively, primitive) element is $x$ such that $\Delta(x) = x \otimes x$ (respectively $\Delta(x) = x \otimes 1 + 1 \otimes x$). In our setup, we shall slightly change this definition to account for the central element $h_0$, by requiring that $\Delta(x) = h_0^{\text{deg}(x)} x \otimes x$ (respectively $\Delta(x) = h_0^{\text{deg}(x)} \otimes x + x \otimes 1$). All results which we shall use remain valid in this slightly changed context}. By this we mean that, on general grounds, the elements $P_{na,nb}$ defined by relation \eqref{eqn:shi} are primitive:
\begin{equation}
\label{eqn:primitive}
\Delta_{b/a}(P_{na,nb}) = P_{na,nb} \otimes 1 + h_0^{na} \otimes P_{na,nb}
\end{equation}
Looking back at the definition of $\Delta_{b/a}$, this precisely implies that $P_{k,d}=P_{na,nb}$ satisfies property \eqref{eqn:gangnam}, which proves the first claim in the lemma. To prove the second claim, take such a quasi-empty triangle and assume that:
$$
\gcd(k,d)=\gcd(k_1,d_1)=1 \qquad \text{and}\qquad \gcd(k_2,d_2) \geq 1
$$
(the other case is taken care of similarly). In this case, we write $(k_2,d_2)=(na,nb)$ for $\gcd(a,b)=1$ and consider the empty triangle with vertices $(0,0), (a,b), (k,d)$. Relation \eqref{eqn:relation} implies:
$$
P_{k,d} = [P_{k-a,d-b},P_{a,b}] \Longrightarrow \Delta(P_{k,d}) = [\Delta(P_{k-a,d-b}), \Delta(P_{a,b})] =
$$

$$
=\left[ \sum_{x=0}^{\left \lfloor \frac ka-1 \right \rfloor} h_0^{xa} P_{k-(x+1)a,d-(x+1)b} \otimes Q_{xa,xb}, P_{a,b} \otimes 1 \right]+... =  
$$

$$
=\sum_{x=0}^{\left \lfloor \frac ka-1 \right \rfloor} h_0^{xa} P_{k-xa,d-xb} \otimes Q_{xa,xb} + (\text{anything}) \otimes \left( \text{slope} < \frac {d_2}{k_2} \right)
$$
The only term with second tensor factor of bidegrees $(k_2,d_2)$ in the above is precisely $h_0^{k_2} P_{k_1,d_1} \otimes Q_{k_2,d_2}$. Since by assumption $\gcd(k_1,d_1)=1$, we have $Q_{k_1,d_1} = \alpha_1 P_{k_1,d_1}$, hence the claim follows.

\end{proof}

\subsection{}
\label{sub:ortho}

According to \cite{BS}, a basis of $\CA_{k,d}$ as a vector space consists of expressions:
$$
P_C = P_{k_1,d_1}...P_{k_t,d_t}
$$
over all collections $C$ of integers such that $k_1+...+k_t = k$, $d_1+...+d_t = d$ and $\frac {d_1}{k_1} \leq ... \leq \frac {d_t}{k_t}$. If a certain number of these ratios are equal, then we order the corresponding $P_{k_i,d_i}$ in increasing order of $k_i$. In the terminology of \cite{BS}, such a collection is determined by a convex path in the lattice $\BZ^2$. It turns out the above basis is orthogonal for the scalar product \eqref{eqn:pair}. \\

\begin{proposition}
\label{prop:ortho}

For all collections $C = \{(k_1,d_1),...,(k_t,d_t) \}$ and $C' = \{(k_1',d_1'),...,(k_s',d_s')\}$ ordered as above, we have:
\begin{equation}
\label{eqn:ortho}
(P_C, P_{C'}) =  \delta_C^{C'}  \prod_{i=1}^t \frac {1}{ \alpha_{\gcd(k_i,d_i)}}
\end{equation}

\end{proposition}

\begin{proof} We will prove the above claim by induction on $k_1+...+k_t$. The bialgebra property of the pairing implies that the LHS of \eqref{eqn:ortho} equals:
$$
\left(P_{k_1,d_1} \otimes P_{k_2,d_2}...P_{k_t,d_t}, \Delta(P_{k_1',d_1'})...\Delta(P_{k_s',d_s'}) \right)
$$
Lemma \ref{lem:imp} implies that the second tensor factor of $\Delta(P_{k_1',d_1'})...\Delta(P_{k_s',d_s'})$ only exists in degrees $(k,d)$ for $\frac dk \leq \frac {d_1'}{k_1'}$. Since $\Delta$ preserves bidegrees, this is equivalent to saying that the first tensor factor only exists in degrees $(k,d)$ for $\frac dk \geq \frac {d_1'}{k_1'}$. Therefore, if $\frac {d_1'}{k_1'} > \frac {d_1}{k_1}$, this tensor factor pairs trivially with $P_{k_1,d_1}$ and the LHS of \eqref{eqn:ortho} vanishes. The same thing would happen if $\frac {d_1'}{k_1'} < \frac {d_1}{k_1}$ since the pairing is symmetric. So let us assume we are in the last remaining case: 
$$
\frac {d_1}{k_1} = \frac {d_1'}{k_1'}
$$
If $k_1<k_1'$ or $k_1>k_1'$, then the same argument would work to prove that the pairing of \eqref{eqn:ortho} is zero. The only other possibility is that $(k_1,d_1)=(k_1',d_1')$. Then the only term in the first tensor factor of $\Delta(P_{k_1',d_1'})...\Delta(P_{k_s',d_s'})$ that has degree small enough to pair non-trivially with $P_{k_1,d_1}$ is:
$$
(P_{k_1',d_1'} \otimes 1)(h_0^{*} \otimes P_{k_2',d_2'}...P_{k_s',d_s'})
$$
\footnote{We do not care about the power of $h_0$, since it pairs trivially with everything} This implies that the LHS of \eqref{eqn:ortho} equals:
$$
(P_{k_1,d_1},P_{k_1,d_1})\cdot (P_{k_2,d_2} ...P_{k_t,d_t},P_{k_2',d_2'}...P_{k_s',d_s'})
$$
which concludes the proof of \eqref{eqn:ortho} by induction. The base of the induction, namely the fact that $(P_{k,d},P_{k,d}) = \frac {1}{\alpha_{\gcd(k,d)}}$, is proved according to Lemma 4.10 of \cite{BS}. 

\end{proof}

\section{Describing $P_{k,d}$ explicitly}
\label{sec:exp}

\subsection{} 
\label{sub:xandy}

Proposition \ref{prop:surj} implies that all shuffle elements can be written as linear combinations of:
\begin{equation}
\label{eqn:form}
z_1^{m_1} * ... * z_1^{m_k} = \sym \left[ z_1^{m_1}...z_k^{m_k} \prod_{1\leq i<j \leq k} \omega \left( \frac {z_i}{z_j} \right) \right]
\end{equation}
as $m_1,...,m_k \in \BZ$. However, another set of elements \footnote{Which will turn out to span the whole shuffle algebra by linear combinations} is also very important: \\

\begin{proposition}

For all $m_1,...,m_k \in \BZ$, the rational function:
\begin{equation}
\label{eqn:defx}
X_{m_1,...,m_k} = \emph{Sym} \left[ \frac {z_1^{m_1}...z_k^{m_k}}{\left(1 - \frac {qz_2}{z_1}\right)...\left(1 - \frac {qz_{k}}{z_{k-1}}\right)} \prod_{1\leq i<j \leq k} \omega \left( \frac {z_i}{z_j} \right) \right]
\end{equation}
is an element of the shuffle algebra $\CA^+$. \\

\end{proposition}

\begin{proof} It is easy to see that the rational function $X_{m_1,...,m_k}$ has the appropriate poles and vanishes when $z_i=z_j$. Since it is symmetric, it must therefore be divisible by $(z_i-z_j)^2$ (note that there are no actual poles when $z_{i-1} = qz_i$, since those factors of the denominator also appear in the numerator of $\omega$). Written in the form \eqref{eqn:wheel}, it is quite easy to see that each summand of the above $\sym$ vanishes when any three of the variables are set to $(1,q_1,q)$ or $(1,q_2,q)$, by an argument akin to the proof of Proposition \ref{prop:0}. Therefore, $X_{m_1,...,m_k}$ is a shuffle element of $\CA^+$. In fact, the denominator: 
$$
\left(1 - \frac {qz_2}{z_1} \right)... \left(1 - \frac {qz_{k}}{z_{k-1}} \right)
$$ 
is maximal \footnote{By maximal we mean that if we added any further linear factors to the denominator, it would cease to satisfy the wheel conditions, and thus fail to be an element of $\CA^+$} so that expressions like the above still satisfy the wheel conditions. Let us remark that the same would hold if we replaced $q$ by $q_1^{-1}$ or $q_2^{-1}$. \\

\end{proof}

\subsection{}

Let us fix a bidegree $(k,d)=(na,nb)$ for $\gcd(a,b)=1$. For any binary string $\e = (\e_1,...,\e_{n-1}) \in \{0,1\}^{n-1}$, consider:
$$
S^{\e}_i = \left\lfloor \frac {id}k \right\rfloor - \e_{i/a}
$$
Naturally, the term $\e_{i/a}$ only appears when $i\in \{a,...,(n-1)a\}$, otherwise we set it equal to 0. \\ 


\begin{proposition}
\label{prop:shuffle}

For any vector $\e$ as above, the shuffle element: 
\begin{equation}
\label{eqn:xkd}
X_{k,d}^\e := X_{S^\e_1-S^\e_0,S^\e_2-S^\e_1,...,S^\e_{k-1}-S^\e_{k-2},S^\e_k-S^\e_{k-1}}
\end{equation}
has slope $\leq \frac dk$, or equivalently, lies in $\CB^{\frac dk}$. Moreover, the shuffle element:
\begin{equation}
\label{eqn:sensei}
\sum_{r+s=n-1}^{r,s \geq 0} q^{s} X_{k,d}^{(0^{r}1^{s})}
\end{equation}
is primitive for the coproduct $\Delta_{d/k}$ on $\CB^{\frac dk}$, where $(0^r1^s) = (\underbrace{0,...,0}_{r\emph{ zeroes}},\underbrace{1,...,1}_{s\emph{ ones}})$. \\

\end{proposition}

\begin{proof} Explicitly, we have:
\begin{equation}
\label{eqn:sigma}
X^\e_{k,d} = \sum_{\sigma \in S(k)} \frac {\prod_{j=1}^k z_{\sigma(j)}^{S^\e_j-S^\e_{j-1}}}{\left(1 - \frac {q z_{\sigma(2)}}{z_{\sigma(1)}} \right)...\left(1 - \frac {q z_{\sigma(k)}}{z_{\sigma(k-1)}}\right)} \prod_{i < j} \omega \left( \frac {z_{\sigma(i)}}{z_{\sigma(j)}}\right)
\end{equation}
To show that the above sum lies in $\CB^{\frac dk}$, we need to multiply the variables $z_{i+1},...,z_k$ by $\xi$ and show that we get something of order no greater than $\frac {(k-i)d}k$ as $\xi \rightarrow \infty$. In fact, we will show that each summand of \eqref{eqn:sigma} has this property. Since $\omega(0) = \omega(\infty) = 1$, the $\omega$ factors do not contribute to the limit. A permutation $\sigma \in S(k)$ is determined by:
$$
A=\sigma^{-1}(\{i+1,...,k\}) \qquad \textrm{ and } \qquad \sigma'\in \textrm{Perm}(A), \quad \sigma''\in \textrm{Perm}(\{1,...,k\}\backslash A)
$$
The set $A$ will be the most important part of the data; it will be of the form:
\begin{equation}
\label{eqn:a}
A  = \{x_1+1,y_1\} \cup \{x_2+1,y_2\} \cup ... \cup \{x_t+1,y_t\}
\end{equation}
where
$$
x_1 < y_1 < x_2 < y_2 < ... < x_t < y_t \in \{0,...,k\} 
$$
are certain indices such that:
\begin{equation}
\label{eqn:chto}
\sum_{j=1}^t (y_j-x_j) = k-i
\end{equation}
The term corresponding to $\sigma$ in \eqref{eqn:sigma} gets a contribution of $\xi$ to the power $\sum_{j=1}^t (S^\e_{y_j} -S^\e_{x_j})$ from the numerator, and $\xi$ to the power $-t+\delta_{x_1}^0$ from the denominator. Therefore, the fact that each summand of $X_{k,d}^\e$ in \eqref{eqn:sigma} has the required degree is equivalent to:
\begin{equation}
\label{eqn:ine}
\sum_{j=1}^t \left( S^\e_{y_j} - S^\e_{x_j}\right) - t + \delta_{x_1}^0 \leq \left \lfloor \frac {(k-i)d}k \right \rfloor
\end{equation}
The above follows from \eqref{eqn:chto} and the simpler inequality:
\begin{equation}
\label{eqn:simpine}
S^\e_y - S^\e_x - 1 +\delta_x^0 = \left \lfloor \frac {yd}k \right \rfloor - \left \lfloor \frac {xd}k \right \rfloor - \e_{y/a} + \e_{x/a} - 1 \leq \left \lfloor \frac {(y-x)d}k \right \rfloor
\end{equation}
which holds for any $x < y \in \{0,...,k\}$, where we make the convention that $\e_0=1$ and $\e_n=0$. This proves that $X_{k,d}^\e \in \CB^{\frac dk}$ for all vectors $\e$. To compute the coproduct $\Delta_{d/k}(X_{k,d}^\e)$, relation \eqref{eqn:radu2} tells us that we need to collect those terms of top order in $\xi$. To do so, we need to trace back through the above inequalities, and see when all of them simultaneously become equalities. We have equality in \eqref{eqn:simpine} when $y = av$ and $x=au$ for some $u<v$ such that $\e(u)=1$ and $\e(v)=0$. This means that sets \eqref{eqn:a} which produce terms of top degree are of the form:
$$
A=\{au_1+1,av_1\} \cup \{au_2+1,av_2\} \cup ... \cup \{au_t+1,av_t\}
$$
for some $u_1,...,u_t \in \e^{-1}(1)$ and $v_1,...,v_t \in \e^{-1}(0)$. Multiplying by $\xi$ the variables corresponding to elements of $A$ and taking the terms of top order in $\xi$ will have the following effect on \eqref{eqn:sigma}: \\

\begin{itemize}

\item the $\omega$ factors between small variables (those not multiplied by $\xi$) and large variables (those multiplied by $\xi$) will converge to 1 \\

\item the factors $\left(1-\frac {qz_{av_j+1}}{z_{av_j}}\right)$ in the denominator will converge to 1 \\

\item the factors $\left(1-\frac {qz_{au_j+1}}{z_{au_j}}\right)$ in the denominator will converge to $(-q)^{-1}\frac {z_{au_j}}{z_{au_j+1}}$ \\

\end{itemize}

Meanwhile, variables from different intervals $i \in \{au_{j}+1,av_j\}$ and $i'\in \{au_{j'}+1,av_{j'}\}$ will interact with each other only through factors $\omega(z_i/z_{i'})$, so the resulting expression will be a shuffle product $*$ of contributions from the individual intervals. We conclude that:
$$
\Delta_{d/k}(X^\e_{k,d}) = \sum_{t \geq 1}  \mathop{\sum^{v_1,...,v_t \in \e^{-1}(0)}}^{u_1,...,u_t \in \e^{-1}(1)}_{0 \leq u_1<v_1<...<u_t<v_t \leq n} (-q)^{-t+\delta_{u_1}^0} \cdot h_0^{a\sum_{i=1}^{t} (v_i-u_{i}-1)}
$$

\begin{equation}
\label{eqn:expdelta}
X^{(\e_1,...,\e_{u_1-1})} * ... * X^{(\e_{v_t+1},...,\e_{n-1})} \otimes X^{(\e_{u_1+1},...,\e_{v_1-1})} * ... * X^{(\e_{u_t+1},...,\e_{v_t-1})} \qquad
\end{equation}
In the above, we write $X^{\e_0} = X^{\e_0}_{ra,rb}$ for $\e_0$ any piece of $r-1$ successive entries of the vector $\e$. By using this formula, we see that:
$$
\Delta_{d/k} \left( \sum_{r+s=n-1}^{r,s \geq 0} q^{s} X^{(0^{r}1^{s})} \right) = \left( \sum_{r+s=n-1}^{r,s \geq 0} q^{s} X^{(0^{r}1^{s})} \right) \otimes 1 + h_0^{na} \otimes \left( \sum_{r+s=n-1}^{r,s \geq 0} q^{s} X^{(0^{r}1^{s})} \right) +
$$

$$
+ \sum_{r+s+r'=n-2}^{r, r', s \geq 0} q^{s} h_0^{(r'+1)a} X^{(0^{r}1^{s})} \otimes X^{(0^{r'})}  - \sum_{r+s+s'=n-2}^{r,s, s'\geq 0} q^{s+s'} h_0^{(s'+1)a}  X^{(0^{r}1^{s})} \otimes X^{(1^{s'})} -
$$

$$
- \sum_{r+s+r'+s'=n-3}^{r,r',s, s'\geq 0} q^{s+s'} h_0^{(r'+s'+2)a}  X^{(0^{r}1^{s})} \otimes X^{(0^{r'})}*X^{(1^{s'})}
$$
where $X^{(0^{-1})} = X^{(1^{-1})} :=1$. To show that the element \eqref{eqn:sensei} is primitive, we need to prove that the last two lines in the above cancel. This is an immediate consequence of the following proposition: \\

\begin{proposition}
\label{prop:drag}

For any $t$, we have: 
\begin{equation}
\label{eqn:id1}
X^{(0^{t-1})} - q^{t-1} \cdot X^{(1^{t-1})}  = \sum^{r, s \geq 0}_{r+s=t-2} q^{s} \cdot X^{(0^{r})} * X^{(1^s)}  
\end{equation}

\end{proposition}

\begin{proof} By definition, for $r,s \geq 0$ which sum up to $t-2$, we have:
$$
X^{(0^{r})} * X^{(1^s)} = 
$$

$$
=  \textrm{Sym} \left[ \frac {\prod_{j=1}^{ta} z_j^{\left \lfloor \frac {jd}k \right \rfloor - \left \lfloor \frac {(j-1)d}k \right \rfloor} \cdot \frac {z_{(r+2)a+1} ...  z_{(t-1)a+1}}{z_{(r+2)a}  ...  z_{(t-1)a}} }{ \left(1 - \frac {qz_2}{z_1}\right)...\left(1 - \frac {qz_{(r+1)a}}{z_{(r+1)a-1}}\right)\left(1 - \frac {qz_{(r+1)a+2}}{z_{(r+1)a+1}}\right)...\left(1 - \frac {qz_{k}}{z_{k-1}}\right)} \prod_{i < j} \omega \left(\frac {z_i}{z_j}\right) \right] = 
$$

$$
= \textrm{Sym} \left[ \frac {\prod_{j=1}^{ta} z_j^{\left \lfloor \frac {jd}k \right \rfloor - \left \lfloor \frac {(j-1)d}k \right \rfloor} \cdot \frac {z_{(r+2)a+1} ...  z_{(t-1)a+1}}{z_{(r+2)a}  ...  z_{(t-1)a}} \left(1 - \frac {q z_{(r+1)a+1}}{z_{(r+1)a}}\right)}{\left(1 - \frac {qz_2}{z_1}\right)...\left(1 - \frac {q z_{k}}{z_{k-1}}\right)} \prod_{i < j} \omega \left(\frac {z_i}{z_j}\right) \right] =
$$

$$
=X^{(0^{r+1}1^{s})} - q X^{(0^r1^{s+1})}
$$
If we multiply this identity by $q^s$ and add up over all $r+s = t-2$, the desired relation follows.

\end{proof}

\end{proof}

\subsection{}
\label{sub:fin}

The above Proposition implies that the shuffle element of \eqref{eqn:sensei} verifies property \eqref{eqn:gangnam}. By the uniqueness of elements satisfying this latter property, the shuffle element of \eqref{eqn:sensei} is proportional to $P_{k,d}$. To figure out the proportionality constant, let us introduce the linear map:
$$
\ph: \CA_{k,d} \longrightarrow \BC(q_1,q_2),
$$

$$
\ph(P) = \left[ P(z_1,...,z_k) \cdot \prod_{1\leq i \neq j \leq k} \frac {z_i-q_1z_j}{z_i-z_j} \right]_{z_i=q_1^{-i}} \cdot \frac {q_1^{\frac {-k^2+kd+d+2k}2}}{(1-q_2)^k} \prod_{i=1}^k \frac {q_1^{i-1}-q_2}{q_1^i-1}
$$
The following Proposition, which will be proved in the Appendix, shows how $\ph$ behaves under the shuffle product: \\

\begin{proposition}
\label{prop:ph}

Given $P \in \CA^+_{k,d}$ and $Q \in \CA^+_{l,e}$, we have: \\
\begin{equation}
\label{eqn:fry}
\ph(P * Q) = \ph(P)\ph(Q) \cdot q_1^{(ld-ke)/2}
\end{equation}
$$$$
\end{proposition}

\subsection{}

We will now proceed to compute the values of the shuffle element \eqref{eqn:sensei} and of $P_{k,d}$ under $\ph$, in order to figure out the proportionality constant between them. \\

\begin{proposition}

For all $k\geq 1$ and $d\in \BZ$ with $\gcd(k,d)=n$, we have:
\begin{equation}
\label{eqn:10}
\ph(P_{k,d})= \frac 1{q_1^{\frac n2}-q_1^{-\frac n2}}
\end{equation}

\end{proposition}

\begin{proof} We will prove this statement by induction on $k$. The base case $k=1$ is trivial. Let us pick an empty triangle with vertices $(0,0), (k_2,d_2), (k,d)$ and apply relation \eqref{eqn:relationx}:
$$
Q_{k,d} = \alpha_1 [P_{k_1,d_1},P_{k_2,d_2}]
$$
Applying $\ph$ and Proposition \ref{prop:ph} to the above relation gives us:
$$
\ph(Q_{k,d}) = \alpha_1 \frac { q_1^{\frac{k_2d_1-k_1d_2}2}}{\left (q_1^{\frac 12}-q_1^{-\frac 12} \right)^2} -\alpha_1 \frac { q_1^{\frac{k_1d_2-k_2d_1}2}}{\left (q_1^{\frac 12}-q_1^{-\frac 12} \right)^2}  = \frac {(q_2-1)(q^{-1}-1)\left(q_1^{\frac n2} - q_1^{-\frac n2}\right)}{(1-q_1^{-1})}
$$
where in the last relation we used Pick's theorem: $k_2d_1-k_1d_2=n$. As a consequence of Proposition \ref{prop:ph}, the linear map $\ph$ is multiplicative on elements of the same slope. We may write $(k,d) = (na,nb)$ and apply $\ph$ to relation \eqref{eqn:shi}:
$$
\exp\left(\sum_{n\geq 1}\alpha_n x^n \ph(P_{na,nb}) \right) = 1 + \sum_{n\geq 1} x^n \ph(Q_{na,nb}) \Longrightarrow
$$

$$
\Rightarrow \sum_{n\geq 1}\alpha_n x^n \ph(P_{na,nb}) = \log\left[1+ \frac {(q_2-1)(q^{-1}-1)}{1-q_1^{-1}} \sum_{n\geq 1} x^n \left(q_1^{\frac n2} - q_1^{-\frac n2}\right) \right] =
$$

$$
=\log\left[1+ \frac {(q_2-1)(q^{-1}-1)q^{\frac 12}_1x}{\left(1-xq_1^{\frac 12}\right)\left(1-xq_1^{- \frac 12}\right)}\right] = \sum_{n\geq 1} \frac {\alpha_nx^n}{q_1^{\frac n2}-q_1^{-\frac n2}} 
$$

\end{proof}

\begin{proposition} 
\label{prop:2}

For all $k\geq 1$ and $d\in \BZ$ with $\gcd(k,d)=n$, and all vectors $\e \in \{0,1\}^{n-1}$, we have:
$$
\ph(X_{k,d}^\e) =  \frac {q_1^{\frac n2-(\#\emph{ of ones in }\e)}}{(q_1-1)^k(1-q_2)^{k-1}}
$$

\end{proposition}

\begin{proof} The shuffle element $X_{k,d}^\e$ is given by the symmetric rational function of \eqref{eqn:sigma}. Together with the definition of $\ph$, this gives us:
$$
\ph(X^\e_{k,d}) = \frac {q_1^{\frac {-k^2+kd+d+2k}2}}{(1-q_2)^k} \prod_{i=1}^k \frac {q_1^{i-1}-q_2}{q_1^i-1} \sum_{\sigma \in S(k)} \frac { \prod_{j=1}^k z_{\sigma(j)}^{S^\e_j-S^\e_{j-1}}}{\left(1 - \frac {q z_{\sigma(2)}}{z_{\sigma(1)}}\right)...\left(1 - \frac {q z_{\sigma(k)}}{z_{\sigma(k-1)}}\right)} 
$$

$$
\prod_{i < j} \frac {(z_{\sigma(j)}-q_1 z_{\sigma(i)})(z_{\sigma(i)}-qz_{\sigma(j)})}{(z_{\sigma(j)}-z_{\sigma(i)})(z_{\sigma(i)}-q_2z_{\sigma(j)})} \Big |_{z_i=q_1^{-i}}
$$
Only one term survives when we evaluate the above at $z_i=q_1^{-i}$, namely the one corresponding to the identity permutation. Therefore, the above gives:
$$
\ph(X^\e_{k,d}) =\frac {q_1^{\frac {-k^2+kd+d+2k}2}}{(1-q_2)^k} \frac { q_1^{\sum_{j=1}^k j(S^\e_{j-1}-S^\e_{j})}}{(1-q_2)^{k-1}} \prod_{i=1}^k \frac {q_1^{i-1}-q_2}{q_1^i-1}  \prod_{i < j} \frac {(q_1^{-j}-q_1^{1-i})(q_1^{-i}-q_2q_1^{1-j})}{(q_1^{-j}-q_1^{-i})(q_1^{-i}-q_2q_1^{-j})} 
$$

$$
=  \frac {q_1^{\frac {-kd+d+k}2+\sum_{j=1}^{k-1} S^\e_j}}{(q_1-1)^k(1-q_2)^{k-1}} =  \frac {q_1^{\frac n2-(\#\textrm{ of ones in }\e)}}{(q_1-1)^k(1-q_2)^{k-1}}
$$
where the last equality follows by Pick's theorem. \\

\end{proof}

\begin{proof} \textbf{of Theorem \ref{thm:main}:} Proposition \ref{prop:2} implies that:
$$
\ph\left(\sum_{r+s=n-1}^{r,s \geq 0} q^{s} X_{k,d}^{(0^{r}1^{s})}\right) = \frac {q_1^{\frac n2} (1-q_2^n)}{(q_1-1)^k(1-q_2)^k}
$$
Comparing this with \eqref{eqn:10} implies that:
\begin{equation}
\label{eqn:sashaa}
P_{k,d} = \frac {(q_1-1)^k(1-q_2)^k}{(q_1^{n}-1)(1-q_2^n)} \sum_{s_0+s_1=n-1}^{s_0,s_1 \geq 0} q^{s_1} X_{k,d}^{(0^{s_0}1^{s_1})}
\end{equation}
Unraveling the definition of $X_{k,d}^\e$ gives us precisely the desired \eqref{eqn:70}. \\

\end{proof}

\subsection{}

For any fixed pair $a,b\in \BN$ with $\gcd(a,b)=1$, Theorem \ref{thm:iso} gives us an isomorphism:
\begin{equation}
\label{eqn:isom}
\Lambda \cong \CB^{\frac ba} = \bigoplus_{n=0}^\infty \CA_{an,bn}^{\frac ba}, \qquad p_n \longrightarrow \frac {(-1)^{n-1} (q_1^{n}-1)(1-q_2^n)}{(q_1-1)^{na}(1-q_2)^{nb}} \cdot P_{na,nb}
\end{equation}
where $\Lambda$ is the bialgebra of symmetric polynomials in infinitely many variables over $\BC(q_1,q_2)$, and $p_n$ denote the power sum functions. This isomorphism sends the natural coproduct on $\Lambda$ to the coproduct $\Delta_{b/a}$ on $\CB^{\frac ba}$, since the $p_n$ and the $P_{na,nb}$ are primitive for the respective coproducts. \\ 


\subsection{}

A quick consequence of \eqref{eqn:sashaa} and \eqref{eqn:isom} is that for each $n$, the vector space $\CB^{\frac ba}_n = \CA_{an,bn}^{\frac ba}$ has a set of linear generators given by the shuffle elements $X_{na,nb}^\e$ of \eqref{eqn:xkd}, as $\e \in \{0,1\}^{n-1}$. There are $2^{n-1}$ such vectors, which is in general greater than the dimension of: 
$$
\CB^{\frac ba}_n \cong \Lambda_n,
$$
which equals the number of integer partitions of $n$. We will now discuss which elements of $\Lambda_n$ correspond to $X_{na,nb}^\e$ under the isomorphism \eqref{eqn:isom}. \\

An important set of symmetric polynomials in $\Lambda$ are the Schur functions $s_\lambda$, defined for any integer partition $\lambda = (\lambda_1 \geq \lambda_2 \geq...)$. To any such parition, we can associate a Young diagram, meaning a set of unit lattice squares in the first quadrant, such that there are $\lambda_1$ of them on the first row, $\lambda_2$ on the second row etc.  Given a pair of integer partitions $\mu \subset \lambda$, by which we mean that the associated Young diagrams are contained one inside the other, one can associate the skew Schur functions $s_{\lambda/\mu}$. The connection between Skew schur functions and Schur functions is given by the Littlewood-Richardson rule (see \cite{Mac}). \\

To any vector $\e \in \{0,1\}^{n-1}$ we can associate a skew diagram $\lambda/\mu$ consisting of $n$ boxes as follows: start from a box on the leftmost column. At each step, depending on whether the corresponding entry of $\e$ is 0 or 1, go either one box down or one box to the right. This traces out a set of boxes, which we then shift vertically so that the bottom box of the set is on the first row. This traces out a particular skew diagram, and we write $s_\e \in \Lambda_n$ for the corresponding skew Schur function. Andrei Okounkov observed that these particular Schur functions correspond to $X_{na,nb}^\e$: \\ 

\begin{proposition}

Under the isomorphism \eqref{eqn:isom}, we have:
\begin{equation}
\label{eqn:visa}
s_\e \longrightarrow (-q)^{r} X_{na,nb}^\e
\end{equation}
where $r$ denotes the number of ones in the vector $\e$. \\

\end{proposition}

\begin{proof} Since $a$ and $b$ are fixed, let us denote $X_{na,nb}^\e$ simply by $X_\e$. We have the following equality:
$$
s_\e s_{\e'} = s_{(\e 0 \e')} + s_{(\e 1 \e')}
$$
where $(\e x \e')$ is the vector obtained by concatenating $\e$ and $\e'$ and putting the digit $x$ between them. This easily follows from the definition of skew Schur functions as sum of monomials indexed by semi-standard Young tableaux, so we leave it as an exercise. We will prove the Proposition by induction on the length of $\e$. By formulas \eqref{eqn:defx} and \eqref{eqn:xkd} that define $X_\e$, it is straightforward to see that:
$$
X_\e * X_{\e'} = X_{(\e 0 \e')} - q X_{(\e 1 \e')}
$$
The proof of this closely follows that of Proposition \ref{prop:drag}, so we will leave it as an exercise. The induction hypothesis then implies that $s_{(\e 0 \e')} + s_{(\e 1 \e')}$ is sent exactly where \eqref{eqn:visa} claims it is sent for all $\e,\e'$. Therefore, to show that \eqref{eqn:visa} sends each $s_\e$ to $X_\e$ it is enough to do so for any given one of them, say for $s_{(0^{n-1})}$. But then by the same argument, it is enough to show that \eqref{eqn:visa} sends:
$$
\sum_{r=0}^{n-1} (-1)^{n-1-r} s_{(0^{n-1-r}1^{r})} \longrightarrow (-1)^{n-1}\sum_{r=0}^{n-1} q^{r} X_{(0^{n-1-r}1^{r})}
$$
The functions $s_{(0^{n-1-r}1^r)}$ are the well-known Schur functions for diagrams of hook shape, and it is well-known that the linear combination in the LHS equals the power sum function $p_n$. Meanwhile, the element on the right is: 
$$
\frac {(-1)^{n-1} (q_1^{n}-1)(1-q_2^n)}{(q_1-1)^{na}(1-q_2)^{nb}} \cdot P_{na,nb},
$$ 
and these two are mapped into each other by \eqref{eqn:isom}, by definition. 

\end{proof}

\section{Appendix}
\label{sec:app}

In the remainder of this paper, we will present proofs to some of the more computational results in this paper: \\

\begin{proof} \textbf{of Proposition \ref{prop:algebra}:} We will prove the following more general statement. Suppose we have shuffle elements $P,P'$ such that the following estimates hold as $\xi \longrightarrow \infty$, for all $i$:
$$
P(\xi z_1,...,\xi z_i,z_{i+1},...,z_k) =\xi^{\left \lfloor \mu i \right \rfloor} \sum P^{(1)}_i(z_1,...,z_i)P^{(2)}_i(z_{i+1},...,z_k) + o(\xi^{\left \lfloor \mu i \right \rfloor})
$$

$$
P'(\xi z_1,...,\xi z_{i'},z_{i'+1},...,z_{k'}) = \xi^{\left \lfloor \mu i' \right \rfloor} \sum P^{(1)'}_i(z_1,...,z_i)P^{(2)'}_i(z_{i+1},...,z_{k'}) + o(\xi^{\left \lfloor \mu i' \right \rfloor})
$$
The above sums indicate that there may be several terms $P^{(1)}_i(\cdot) P^{(2)}_i(\cdot)$ appearing for each $i$. Then we can write:
$$
(P*P')(\xi z_1,...,\xi z_j,z_{j+1},...,z_{k+k'}) = \sum^{i+i'=j}_{\left \lfloor \mu i \right \rfloor + \left \lfloor \mu i' \right \rfloor = \left \lfloor \mu j \right \rfloor}  \xi^{\left \lfloor \mu j \right \rfloor}
$$

\begin{equation}
\label{eqn:generalclaim}
 \sum (P^{(1)}_i*P^{(1)'}_{i'})(z_1,...,z_j) (P^{(2)}_i*P^{(2)'}_{i'})(z_{j+1},...,z_{k+k'}) + o(\xi^{\left \lfloor \mu j \right \rfloor})
\end{equation}
by the definition of the product in \eqref{eqn:mult}, and because: 
$$
\lim_{\xi \rightarrow \infty} \omega(\xi) = \lim_{\xi \rightarrow 0} \omega(\xi) = 1
$$
Recalling the definition of the space $\CA^\mu$, this is precisely what we needed to prove. \\

\end{proof}

\begin{proof} \textbf{of Proposition \ref{prop:copcheck}:} The coassociativity of the coproduct is the statement that $(\Delta \otimes \text{Id})\circ \Delta = (\text{Id} \otimes \Delta)\circ \Delta$. This relation indeed holds on any shuffle element $P\in \CA^+_k$, since both sides are equal to:
$$
\sum_{0\leq i \leq j \leq k} \frac {\left(\prod_{b>i} h(z_b) \otimes \prod_{c>j} h(z_c) \otimes 1 \right)\cdot  P(z_1,...,z_i \otimes z_{i+1},...,z_j \otimes z_{j+1},...,z_k)}{\prod^{a\leq i}_{i<b\leq j} \omega(z_b/z_a)\prod^{a\leq i}_{j<c} \omega(z_c/z_a)\prod^{i<b\leq j}_{j<c} \omega(z_c/z_b)}
$$
expanded in non-negative powers of $z_b/z_a$ and $z_c/z_b$. As before, the variables $z_a$ move between the tensor product signs such that $z_1,...,z_i$ sit in the first tensor factor, $z_{i+1},...,z_j$ sit in the second tensor factor, and $z_{j+1},...,z_k$ sit in the third tensor factor. \\

We still need to prove that $\Delta$ respects the multiplication in $\CA^\geq$. One of the relations one needs to check is that $\Delta$ respects relation \eqref{eqn:relhp} between shuffle elements in $\CA^+$ and Cartan elements in $\CA^0$. This is a straightforward exercise, and we leave it to the reader. The remaining relation is more interesting and non-trivial, namely the fact that:
$$
\Delta(P*Q) = \Delta(P)*\Delta(Q)
$$
for any shuffle elements $P \in \CA_k$ and $Q\in \CA_l$. We will proceed to prove this relation. By definition, the LHS equals:
$$
P*Q = \sum_{\{1,...,k+l\}=A\sqcup B} P(z_A)Q(z_B)\omega\left( \frac {z_A}{z_B} \right)
$$
where the sum goes over all partitions with $|A|=k$, $|B|=l$. Given a set of indices $A=\{a_1,...,a_k\}$, we use above the shorthand notation $P(z_A)=P(z_{a_1},...,z_{a_k})$ to unburden our notation. For any $i\in \{0,...,k+l\}$, we will denote:
$$
A_1 = A\cap \{1,...,i\}, \qquad A_2 = A \cap \{i+1,...,k+l\}
$$

$$
B_1 = B \cap \{1,...,i\}, \qquad B_2 = B \cap \{i+1,...,k+l\}
$$
Then the coproduct \eqref{eqn:coproduct} of $P*Q$ is given by:
$$
\Delta(P*Q) = \sum_{i=0}^{k+l} \sum_{\{i+1,...,k+l\} = A_2 \sqcup B_2}^{\{1,...,i\}=A_1 \sqcup B_1} \frac {h(z_{A_2}) h(z_{B_2}) P(z_{A_1} \otimes z_{A_2})Q(z_{B_1} \otimes z_{B_2}) \omega \left( \frac {z_{A_1 \sqcup A_2}}{z_{B_1 \sqcup B_2}} \right)}{\omega \left( \frac {z_{A_2 \sqcup B_2}}{z_{A_1 \sqcup B_1}} \right)}
$$
We can commute $h(z_{B_2})$ past $P(z_{A_1})$, and the price we pay is a ratio of $\omega$'s as in \eqref{eqn:relhp}. We conclude that $\Delta(P*Q)$ equals:
$$
\sum_{i=0}^{k+l} \sum_{\{i+1,...,k+l\} = A_2 \sqcup B_2}^{\{1,...,i\}=A_1 \sqcup B_1} \frac {h(z_{A_2}) P(z_{A_1} \otimes z_{A_2})}{\omega \left( \frac {z_{A_2}}{z_{A_1}} \right)}  \cdot  \frac {h(z_{B_2}) Q(z_{B_1} \otimes z_{B_2})}{\omega \left( \frac {z_{B_2}}{z_{B_1}} \right)}  \cdot \omega \left( \frac {z_{A_1}}{z_{B_1}} \right)\omega \left( \frac {z_{A_2}}{z_{B_2}} \right)
$$
The RHS is precisely $\Delta(P)*\Delta(Q)$, thus concluding the proof. 

\end{proof}

\begin{proof} \textbf{of Proposition \ref{prop:well}:} Let us first prove the symmetry of the pairing \eqref{eqn:normal}, i.e. that:
$$
\left(\sym \left[ z_1^{n_1}...z_k^{n_k} \prod_{1\leq i < j \leq k} \omega(z_i/z_j) \right], \sym \left[ z_1^{m_1}...z_k^{m_k} \prod_{1\leq i < j \leq k} \omega(z_i/z_j) \right] \right) 
$$
is symmetric in $m$'s and $n$'s. By the definition of the normal ordered integral, the above equals $\frac {c_1^k c_2^d} {\alpha_1^k}$ times the integral:
\begin{equation}
\label{eqn:11}
\int_{|u_1|\ll...\ll|u_k|} \sum_{\sigma \in S(k)} u_1^{n_1}...u_k^{n_k} u_{\sigma(1)}^{-m_1} ... u_{\sigma(k)}^{-m_k} \mathop{\prod_{i>j}}_{\sigma^{-1}(i)<\sigma^{-1}(j)} \frac {\omega(u_j/u_i)}{\omega(u_i/u_j)} Du_1...Du_k
\end{equation}
If we change variables to $v_i = u_{\sigma(i)}$, the above becomes:
$$
\sum_{\sigma \in S(k)} \int_{|v_{\sigma^{-1}(1)}|\ll...\ll|v_{\sigma^{-1}(k)}|} v_{\sigma^{-1}(1)}^{n_1}...v_{\sigma^{-1}(k)}^{n_k} v_{1}^{-m_1} ... v_{k}^{-m_k} \mathop{\prod_{i<j}}_{\sigma(i)>\sigma(j)} \frac {\omega(v_j/v_i)}{\omega(v_i/v_j)} Dv_1...Dv_k
$$
If we write $\tau = \sigma^{-1}$ and $w_i = v_i^{-1}$, the above becomes:
$$
\sum_{\tau \in S(k)} \int_{|w_{\tau(k)}|\ll...\ll|w_{\tau(1)}|} w_{1}^{m_1} ... w_{k}^{m_k} w_{\tau(1)}^{-n_1}...w_{\tau(k)}^{-n_k}  \mathop{\prod_{i>j}}_{\tau^{-1}(i)<\tau^{-1}(j)} \frac {\omega(w_j/w_i)}{\omega(w_i/w_j)} Dw_1...Dw_k
$$
If $i>j$ are such that $\tau^{-1}(i)>\tau^{-1}(j)$, then we can move the contours so as to change $|w_i|\ll|w_j|$ to $|w_i|\gg|w_j|$ (the reason is because we are not hindered by the poles of the fraction). Therefore, the above equals \eqref{eqn:11} with $m$'s and $n$'s switched, so we conclude that the pairing is symmetric. \\

The pairing is also easily seen to be non-degenerate in the second variable: if \eqref{eqn:normal} vanishes for all $n_1,...,n_k\in \BZ$, then we conclude that $P=0$. We will use this to show that the pairing is well-defined in the first variable. We need to prove that for any $P_1,...,P_l$ of the form \eqref{eqn:form}: 
$$
\sum_{i=1}^l c_iP_i = 0 \Longrightarrow \sum_{i=1}^l c_i (P_i,P)=0 \quad \forall P \in \CA^+
$$
By Proposition \ref{prop:surj}, we can write $P = \sum d_{j} P'_j$ for each $P'_j$ of the form $z^{m_1} * ... * z^{m_k}$. Symmetry implies that
$$
\sum_i \sum_j c_i d_j (P_i,P_j') = \sum_i \sum_j c_i d_j (P_j',P_i) = \sum_j d_j \left(P_j',\sum_i c_i P_i \right)=0 
$$
because the pairing, as defined in \eqref{eqn:normal} is certainly additive and well-defined in the second variable. We have therefore shown that $(\cdot,\cdot)$ is a well-defined symmetric pairing on $\CA^+$, and now we must prove that it satisfies the bialgebra property \eqref{eqn:bialg}. By Proposition \ref{prop:surj}, it is enough to show that:
$$
\left( \sym \left[z_1^{n_1}...z_k^{n_k}z_{k+1}^{m_1}...z_{k+k'}^{m_{k'}} \prod_{i<j} \omega(z_i/z_j) \right], P \right) =  
$$

$$
=\left( \sym \left[z_1^{n_1}...z_k^{n_k}\prod_{i<j} \omega(z_i/z_j) \right] \otimes \sym \left[z_1^{m_1}...z_k^{m_{k'}}\prod_{i<j} \omega(z_i/z_j) \right], \Delta(P) \right)
$$
This is immediate from \eqref{eqn:coproduct} and \eqref{eqn:normal}. To extend the pairing from $\CA^+$ to $\CA^\geq$, we need to show that it is compatible via the bialgebra pairing with relations \eqref{eqn:relhp}. This is a straightforward exercise, and we leave it to the interested reader. 

\end{proof}

\begin{proof} \textbf{of Proposition \ref{prop:ph}:} Let us write:
$$
F(k,d):= \frac {q_1^{\frac {-k^2+kd+d+2k}2}}{(1-q_2)^k} \prod_{i=1}^k \frac {q_1^{i-1}-q_2}{q_1^i-1}
$$
Then the LHS of \eqref{eqn:fry} equals:
$$
 \sym\left[ P(z_1,...,z_k)Q(z_{k+1},...,z_{k+l}) \prod_{1\leq i \neq j \leq k+l} \frac {z_i-q_1z_j}{z_i-z_j} \prod^{i \leq k}_{j>k+1} \frac {(z_i-z_j)(z_i-qz_j)}{(z_i-q_1z_j)(z_i-q_2z_j)} \right]_{z_i = q_1^{-i}} \cdot
$$ 

$$
 F(k+l,d+e) = \sym\left[ P(z_1,...,z_k)Q(z_{k+1},...,z_{k+l}) \prod_{1\leq i \neq j \leq k} \frac {z_i-q_1z_j}{z_i-z_j} \prod_{k+1\leq i \neq j \leq k+l} \frac {z_i-q_1z_j}{z_i-z_j} \right.
$$

$$
\left. \prod^{i \leq k}_{j>k+1} \frac {(z_i-qz_j)(z_j-q_1z_i)}{(z_i-q_2z_j)(z_j-z_i)} \right]_{z_i = q_1^{-i}} \cdot  F(k+l,d+e)
$$
The last factor in the numerator means that the only terms which do not vanish in the above $\sym$ are those which put the variables $\{z_1,...,z_k\}$ before the variables $\{z_{k+1},...,z_{k+l}\}$, leaving the above equal to:
$$
\sym\left[ P(z_1,...,z_k)\prod_{1\leq i \neq j \leq k} \frac {z_i-q_1z_j}{z_i-z_j} \right]_{z_i=q_1^{-i}} \sym\left[ Q(z_{k+1},...,z_{k+l})\prod_{k+1\leq i \neq j \leq k+l} \frac {z_i-q_1z_j}{z_i-z_j} \right]_{z_i=q_1^{-i}} 
$$

$$
 \prod^{1\leq i\leq k}_{k+1 \leq j \leq k+l} \frac {(q_1^{-i}-qq_1^{-j})(q_1^{-j}-q_1^{1-i})}{(q_1^{-i}-q_2q_1^{-j})(q_1^{-j}-q_1^{-i})}  \cdot F(k+l,d+e) = 
$$ 

$$
= \ph(P) \ph(Q) q_1^{-ek+kl} \prod^{1\leq i\leq k}_{k+1 \leq j \leq k+l} \frac {(q_1^{j-i-1}-q_2)(q_1^{j-i+1}-1)}{(q_1^{j-i}-q_2)(q_1^{j-i}-1)} \cdot \frac {F(k+l,d+e)}{F(k,d)F(l,e)}
$$ 
The products in the above expression cancel the $F$'s, leaving the desired RHS of \eqref{eqn:fry}. 

\end{proof}

\end{document}